\documentclass[11pt]{amsart}

\usepackage{amsmath,amssymb,amsfonts,amsthm,mathrsfs}
\usepackage{hyperref}
\usepackage{verbatim}

\hypersetup{
  colorlinks=true,
  linkcolor=blue,
  citecolor=blue,
  urlcolor=blue
}

\newtheorem{theorem}{Theorem}[section]
\newtheorem{lemma}[theorem]{Lemma}
\newtheorem{proposition}[theorem]{Proposition}
\newtheorem{corollary}[theorem]{Corollary}

\theoremstyle{definition}
\newtheorem{definition}[theorem]{Definition}

\newtheorem{remark}[theorem]{Remark}

\newcommand{\R}{\mathbb{R}}

\title[Degeneracy Set Comparison and Free Boundary Estimates]{Degeneracy Set Comparison Principle and Free Boundary Estimates for H\'enon-type Infinity Laplace equations}

\author{Yantian Chen}
\address[Yantian Chen]{School of Mathematics, Southeast University,
Nanjing 211189, P.R. China}
\email{213230018@seu.edu.cn}

\author{Feida Jiang$^*$}
\address[Feida Jiang]{School of Mathematics and Shing-Tung Yau Center of Southeast University,
Southeast University, Nanjing 211189, P.R. China;
Shanghai Institute for Mathematics and Interdisciplinary Sciences,
Shanghai 200433, P.R. China}
\email{jiangfeida@seu.edu.cn}

\thanks{$^*$Corresponding author.}



\begin{document}

\begin{abstract}
In this work, we study nonnegative viscosity solutions to H\'enon-type equations driven by the infinity-Laplacian with a degenerate weight, strong absorption and an additional source term. We focus on free boundary points lying in the degeneracy set of the weight. We prove the degeneracy set comparison principle, which yields uniqueness within each slice determined by the value on the degeneracy set of the weight, although global uniqueness is not available in general. We establish sharp improved regularity estimates near free boundary points, with the intrinsic growth rate $r^{\frac{4+\alpha}{3-m}}$. Finally, we obtain a matching non-degeneracy estimate at free boundary points; the property holds in the inhomogeneous case with suitable condition. Consequently, solutions detach from their zero phase at this rate.
\end{abstract}

\maketitle

\section{Introduction}

This current work is devoted to the study of the inhomogeneous H\'enon-type equation driven by infinity Laplace operator featuring a degenerate weight, strong absorption, and an additional source term. More precisely, we consider the associated Dirichlet problem:
\begin{equation}
\label{eq:general-dirichlet}
\begin{cases}
\Delta_\infty u=f(|x|,u)+h(x) & \text{in }B_1,\\
u=g & \text{on }\partial B_1.    \end{cases}
\end{equation}
where $B_1 \subset \R^n$, $n \geq 2$, denotes the unit $n$-dimensional ball centered at the origin, $g\in C(\partial B_1)$ is nonnegative, and $h\in C(B_1)\cap L^\infty(B_1)$, $h\le0$.

Following the H\'enon-type framework in \cite{BezerraDaSilvaNascimentoSa2024}, we impose a scaling condition on the absorption term $f$. More precisely, for all $(x, t) \in B_1 \times \mathfrak{I}$, $r, s \in (0, 1)$, where $\mathfrak{I} \subset \mathbb{R}$ is an interval, we assume that there exist a universal constant $\mathrm{c}_n > 0$ and an $f_0 \in L^\infty(B_1)$ such that
\begin{equation*}
|f(r|x|, s t)| \leq \mathrm{c}_n r^{\alpha} s^m \|f_0\|_{L^\infty(B_1)} \quad \text{for} \quad 0 < m < 3 \quad  \text{and} \quad  \alpha \in \left[0, \infty\right).
\end{equation*}

Throughout the paper, we will focus on the H\'enon-type as the special case 
\[
f(|x|,t)=|x|^{\alpha}t_+^m.
\]

Therefore, the principal prototype we focused on in the sequel is
\begin{equation}
\label{eq:main}
    \Delta_\infty u=|x|^\alpha u_+^m+h(x)
    \quad\text{in }B_1.
\end{equation}

The problem above involves the infinity Laplace operator, which can be written as
\[
\Delta_{\infty}u
:=
\sum_{i,j=1}^n\partial_i u\,\partial_{ij}u\,\partial_j u
=
(Du)^T D^2u\,Du.
\]

The study of the infinity Laplacian originates from Aronsson's work on absolutely minimizing Lipschitz extensions (AMLE) \cite{Aronsson1967}. The basic question asks for a Lipschitz function $u$ with $u=g$ on $\partial\Omega$ that minimizes the maximal slope. A function is absolutely minimizing if for every $V\Subset\Omega$ and every Lipschitz function $v$ satisfying $v=u$ on $\partial V$, $\|Du\|_{L^{\infty}(V)}\leq \|Dv\|_{L^{\infty}(V)}$. Since minimal Lipschitz extensions need not be unique, Aronsson introduced the stronger notion of an absolutely minimizing Lipschitz extension. For smooth functions, the corresponding Euler equation is $\Delta_{\infty}u=0$. However, since the operator is highly degenerate, smooth solutions cannot in general be expected. Jensen later proved that absolutely minimizing Lipschitz extensions are precisely the solutions of $\Delta_{\infty}u=0$ in the viscosity sense and established the uniqueness for the corresponding Dirichlet problem \cite{Jensen1993}.

Although the existence and uniqueness for the homogeneous Dirichlet problem of $\Delta_{\infty}u=0$ is well established, the regularity of the solution is much more delicate and subtle. 

The classical Aronsson example
\[
u(x,y)=|x|^{4/3}-|y|^{4/3}
\]
suggests that $C^{1,1/3}$ is the expected optimal regularity. In dimension two, Evans and Savin proved that infinity harmonic functions are of class $C^{1,\alpha}$, for some $\alpha>0$, see \cite{EvansSavin2008}. In arbitrary dimensions, Evans and Smart obtained everywhere differentiability for infinity harmonic functions \cite{EvansSmart2011}.

Very recently, Xu proved a variety of new structural and regularity results of infinity-harmonic functions in two dimensions in a preprint \cite{Xu2026}, including that infinity-harmonic functions in domains of $\mathbb{R}^2$ belong to $C_{\mathrm{loc}}^{1,1/3}$.

Beyond homogeneous setting, the inhomogeneous infinity Laplace equation has also received much attention, which is more delicate. Lu and Wang studied the Dirichlet problem and proved the existence and uniqueness of viscosity solutions of that provided the inhomogeneous term $f$ has a strict fixed sign, i.e. either $\inf f > 0$ or else $\sup f < 0$ in \cite{LuWang2008}. They also showed non-uniqueness example when the forcing term is allowed to change sign. It is also known that viscosity solutions of the infinity Laplacian equation with bounded continuous inhomogeneous term are locally Lipschitz continuous. Furthermore, Lindgren showed that the solutions are differentiable everywhere provided $f\in C^{1}$, see \cite{Lindgren2014}. In \cite{KochZhangZhou2019}, Koch, Zhang and Zhou established sharp Sobolev regularity estimates for $|Du|^\alpha$, given that $f\in BV_{\mathrm{loc}}(\Omega)\cap C(\Omega)$ and $|f|>0$.

Despite the advances obtained, the regularity theory for inhomogeneous infinity Laplace equations remains far from complete. In recent years, a particularly relevant class of inhomogeneous problem regarding reaction-diffusion models with strong absorption has attracted considerable attention. Such problems usually present a phenomenon of plateaus, or dead cores, where a nonnegative solution vanishes identically. Remarkably, the additional absorption structure may lead to improved regularity and non-degeneracy properties as viscosity solutions detach from their zero sets.

In this direction, Ara\'ujo, Leit\~ao and Teixeira~\cite{AraujoLeitaoTeixeira2016} studied the reaction--diffusion model
\[
\Delta_\infty u = \lambda (u_+)^\gamma, \qquad 0 \leq \gamma < 3,
\]
initially with a constant Thiele modulus $\lambda > 0$. They established existence and uniqueness of nonnegative viscosity solutions and obtained sharp geometric regularity estimates near the free boundary $\partial\{u>0\}$. More precisely, solutions detach from their dead-core regions with the characteristic rate
\[
\mathrm{dist}\left(x,\partial\{u>0\}\right)^{\frac{4}{3-\gamma}},
\]
together with a matching non-degeneracy estimate. Their results reveal that the strong absorption structure may produce an improved regularity near the free boundary, compared with $C^{1,\alpha}$ for general inhomogeneous infinity Laplace equations. They also pointed out that The regularity estimate extends to bounded spatially dependent weight function $\lambda=\lambda(x)$.

The model was subsequently extended by Lin and Liu in \cite{LinLiu2022}, who considered the inhomogeneous problem
\[
\Delta_\infty u - \lambda(x)(u_+)^\gamma = f(x) \quad \text{in } \Omega,
\]
where $\lambda$ is a spatially dependent weight function, $\inf_{\Omega} \lambda>0$ and $f\leq0$. They proved existence and uniqueness of non-negative viscosity solutions and, under additional assumptions on the source term $f$, recovered the $C^{\frac{4}{3-\gamma}}$ regularity across the free boundary as well as the corresponding non-degeneracy property. They also established a stability result showing that solutions converge uniformly to the solution of the homogeneous strong-absorption problem as the inhomogeneous term tends uniformly to zero under the assumption $\inf_{\Omega}\lambda>0$.

More recently, Bezerra J\'unior, da Silva, Nascimento and S\'a \cite{BezerraDaSilvaNascimentoSa2024} investigated the effect of a spatially degenerate H\'enon-type weight by considering, as a prototype,
\[
\Delta_\infty u = |x|^\alpha u_+^m \quad \text{in } B_1, \qquad 0 \leq m < 3.
\]
The weight function $|x|^\alpha$ introduces an additional degeneracy at the origin and changes the natural growth scale of the problem. They obtained the sharp exponent
\[
\beta = \frac{4+\alpha}{3-m},
\]
together with improved regularity estimates of order $\beta$. Thus, the vanishing of the H\'enon weight at the origin changes the natural free boundary scale from $r^{\frac{4}{3-m}}$ to $r^{\frac{4+\alpha}{3-m}}$. 

In a broader framework, Wang and Jiang \cite{WangJiang2026} studied strong-absorption problems governed by degenerate or singular fully nonlinear elliptic operators. Their results include improved regularity and non-degeneracy along free boundaries, measure estimates, Liouville-type properties and blow-up analysis, and also extend the study of Hardy--H\'enon-type equations to this degenerate or singular fully nonlinear setting.

For the non-degeneracy analysis of the H\'enon-type problem, they further restricted their attention to the class of nonnegative \emph{limiting solutions} \cite{WangJiang2026}. More precisely, a viscosity solution $u$ is called a limiting
solution if it can be obtained as a locally uniform limit of nonnegative viscosity solutions $u_j$ to the penalized problems
\begin{equation*}
    |Du_j|^p F(D^2u_j,x)
    =
    f(|x|,u_j)+\frac{1}{j}.
\end{equation*}

The idea of limiting solutions is already implicit in earlier \cite{BezerraDaSilvaNascimentoSa2024}; Wang and Jiang made this viewpoint explicit by formalizing the class of limiting solutions and using it systematically in their non-degeneracy analysis. This provides a natural approximation framework for studying non-degeneracy at critical points. 

These developments show that the improved free-boundary behavior is not restricted to the standard infinity Laplacian, but rather reflects a more general interaction between degenerate diffusion and strong absorption.

Nevertheless, some questions remain to be answered. The preceding works leave a gap between the inhomogeneous strong absorption theory and the degenerate H\'enon-type setting. Lin and Liu \cite{LinLiu2022} consider a source term but assume that the weight is strictly positive, whereas Bezerra J\'unior, da Silva, Nascimento and S\'a \cite{BezerraDaSilvaNascimentoSa2024} consider a degenerate weight but focus on the homogeneous case. This raises the question of how the degeneracy of the H\'enon weight and the additional source term would influence the free boundary growth and the whole framework structure. To the best of our knowledge, currently general global uniqueness theory for H\'enon-type Dirichlet problem is not available. Indeed, the available comparison and uniqueness results require the weight to remain strictly positive, bringing us genuine obstacles. Moreover, in the discussion of non-degeneracy analysis, the existing literature relies on the restriction of limiting solutions. We are also curious whether limiting solutions are intrinsic to the non-degeneracy phenomenon in the degenerate weight case.

Our main results are as follows.

\begin{theorem}[\textbf{Degeneracy-set Comparison Principle}]
\label{thm:de-set-comparison}
Suppose $\Omega$ is a bounded domain in $\mathbb{R}^n$. Let $h\in C(\Omega)\cap L^\infty(\Omega)$ with $h\le0$, $\lambda(x)\in C(\Omega)$ with $\lambda(x)\ge0$ and $0<m<3$. Let $u,v \in C^0(\overline{\Omega})$ be non-negative functions such that $u$ is a viscosity subsolution and $v$ is a viscosity supersolution of 
\[
\Delta_\infty w-\lambda(x)w_+^m=h(x) \qquad \text{in } \Omega,
\]
that is 
\[
\Delta_\infty u-\lambda(x)u_+^m\ge h(x),
\]
and
\[
\Delta_\infty v-\lambda(x)v_+^m\le h(x)
\]
in the viscosity sense. Denote the degeneracy set of weight function $F:=\{x\in\Omega:\lambda(x)=0\}$, if $u\le v$ on $\partial\Omega\cup F$, then $u\le v$ in $\Omega$.
\end{theorem}

\begin{theorem}[\textbf{Improved regularity at the degenerate point on the free boundary}]
\label{thm:regularity}
Let \(u\in C(B_1)\) be a nonnegative viscosity solution of
\begin{equation}\label{eq:source-model}
    \Delta_\infty u=|x|^\alpha u^m+h(x)
    \quad\text{in }B_1,
\end{equation}
where
$
\alpha>0, 
0<m<3,
h\in C(B_1)\cap L^\infty(B_1)$, and
$
h\le0.
$

Assume that $0\in\partial \{u>0\}$.
Then there exists a constant \(C>0\), depending only on $\alpha$, $m$, $\|u\|_{L^{\infty}(B_1)}$ and the dimension, such that
\begin{equation}
\label{eq:regularity}
    \sup_{x\in B_r(0)}u(x)
    \le
    C\cdot r^{\frac{4+\alpha}{3-m}}
\end{equation}
for every \(0<r<1/3\). In particular, \(u\) is differentiable at the origin and $Du(0)=0$.
\end{theorem}

\begin{theorem}[\textbf{Non-degeneracy at the degenerate free-boundary point of homogeneous model}]
\label{thm:non-de-homo}
Let $\alpha>0$, $m\in (0,3)$, $u\in C^0(B_1)$ be a nonnegative viscosity solution of
\[
\Delta_\infty u=|x|^\alpha u_+^m
\qquad\text{in }B_1.
\]

Assume that $0\in\partial\{u>0\}$.
Then, there exists $r^*>0$ such that for all $r\in(0,r^*)$ such that
$B_r(0)\subset B_1(0)$, we have
\[
\sup_{\partial B_r(0)} u(x)
\ge
\left(
\frac{(3-m)^4}
{(4+\alpha)^3(1+\alpha+m)}
\right)^{\frac{1}{3-m}}
\, r^{\frac{4+\alpha}{3-m}}.
\]
\end{theorem}

\begin{theorem}[\textbf{Non-degeneracy at the degenerate free-boundary point of inhomogeneous model}]
\label{thm:non-de-inhomo}
Let $\alpha>0$, $m\in (0,3)$ and $u\in C^0(B_1)$ be a nonnegative viscosity solution to \eqref{eq:main}, where $h\in C(B_1)\cap L^\infty(B_1)$ with $h\le0$. Assume $0\in\partial\{u>0\}$ and that $0\in \overline{\{u>0\}\cap\{h=0\}}$.

Fix $\theta\in(0,1)$ and choose $s>0$ such that $s<
\left[
\dfrac{(3-m)^4(1-\theta)^\alpha}
{64(1+m)}
\right]^{1/(3-m)}$ and further assume that for any $x_0\in\{u>0\}\cap\{h=0\}$, there holds
\begin{equation}
\label{cond:non-de-h1}
0\ge h(x)\ge
-s^m
\left[
(1-\theta)^\alpha
-
\frac{64(1+m)}{(3-m)^4}s^{3-m}
\right]
\cdot
|x_0|^{\frac{3\alpha}{3-m}}
|x-x_0|^{\frac{4m}{3-m}}
\tag{H1}
\end{equation}
for every $x\in B_{\theta|x_0|}(x_0)$.

Moreover, for every $\eta\in(0,1)$, there exist $r_\eta\in(0,1)$ and sufficiently small $\mu_0=\mu_0(\alpha, m, \theta, s, \eta)>0$ such that 
\begin{equation}
\label{cond:non-de-h2}
0\ge h(x)\ge
-\mu_0 |x|^{\frac{3\alpha+4m}{3-m}}
\tag{H2}
\end{equation}
for every $x\in B_{r_\eta}$. Then for every $r\in(0,r_\eta)$ we have
\begin{equation*}
\label{non-de-h-sup}
\sup_{\partial B_r}u
\ge
\eta 
\left(
\frac{(3-m)^4}
{(4+\alpha)^3(1+\alpha+m)}
\right)^{\frac1{3-m}}
\cdot
~r^{\frac{4+\alpha}{3-m}}.
\end{equation*}
\end{theorem}

\begin{remark}
A more convenient version of \eqref{cond:non-de-h2} can be
\begin{equation*}
\tag{H2'}
h(x)=o\left(|x|^{\frac{3\alpha+4m}{3-m}}\right),
\end{equation*}
as $|x|\to0$ and the result holds.
\end{remark}

\subsection{Organization of the paper}
The remainder of the paper is organized as follows. In Section 2, we introduce the notation and establish the basic properties of the Dirichlet problem that will be used throughout the paper. In particular, we discuss the viscosity framework and prove the existence, non-negativity, and $L^\infty$-bounds of solutions, together with the auxiliary comparison and regularity estimates needed in the sequel.

In Section 3, we establish the degeneracy set comparison principle. Then we derive the corresponding slice uniqueness property. 

Section 4 is devoted to the improved regularity of solutions near free-boundary points. Adapting the iteration argument from the homogeneous H\'enon-type setting to the present inhomogeneous problem, we obtain the sharp upper growth estimate with exponent $\frac{4+\alpha}{3-m}$.

Section 5 is devoted to the non-degeneracy of solutions near the free boundary for both homogeneous and inhomogeneous case from various perspectives. Combined with the upper growth estimate obtained in Section 4, this yields the optimal growth rate dictated by the interaction between the strong absorption and the degenerate H\'enon-type weight.

Finally, in Section 6, we conclude with several remarks and open questions concerning possible extensions of the present analysis.

\section{Preliminary Results}

In this section, we first claim the definition of  viscosity solution to the H\'enon-type equation driven by infinity Laplace operator. Throughout this paper, we work with the prototype $f(|x|,u(x))=|x|^{\alpha}u_{+}^m$,
\begin{equation*}
    \Delta_\infty u=|x|^\alpha u_+^m+h(x).
\end{equation*}

Then we introduce some fundamental lemmas, propositions and theorems that would be useful in the following argument.

\begin{definition}[\textbf{Viscosity solution}]
We say that a function $u\in C^0(\Omega)$ is a viscosity
subsolution (respectively, supersolution) of the PDE
\begin{equation}
\label{eq:viscosity-def}
\Delta_\infty u(x)
=
|x|^\alpha (u(x))_+^m+h(x)
\qquad\text{in }\Omega
\end{equation}
if, for every $\varphi\in C^2(\Omega)$ such that
$u-\varphi$ has a local maximum (resp. minimum) at some
$x_0\in\Omega$, then
\[
\Delta_\infty\varphi(x_0)
\ge
|x_0|^\alpha\bigl(u(x_0)_+\bigr)^m+h(x_0),
\qquad
(resp. \le).
\]

A function $u\in C^0(\Omega)$ is called a viscosity
solution of \eqref{eq:viscosity-def} if it is both a viscosity subsolution and a viscosity supersolution.

Furthermore, if $u$ is a viscosity solution of the equation \eqref{eq:viscosity-def} and $u(x) = g(x)$ on $\partial\Omega$, we say $u$ is a viscosity solution of the corresponding Dirichlet problem.
\end{definition}

\begin{lemma}[\textbf{Comparison Principle -}
{\cite[Lemma~4.1]{BhattacharyaMohammed2011}}]
\label{lem:boundary-comparison}
Let $f_i:\Omega\times\mathbb{R}\to\mathbb{R}$,
$i=1,2$, be continuous, and let
$u,v\in C(\overline{\Omega})$ satisfy, in the viscosity sense,
\[
\Delta_\infty u \geq f_1(x,u),
\qquad
\Delta_\infty v \leq f_2(x,v)
\quad\text{in }\Omega.
\]
Furthermore, assume that either $f_1(x,t)$ or $f_2(x,t)$ is
non-decreasing in $t$, and that
$f_1(x,t)>f_2(x,t)$ for all $(x,t)\in\Omega\times\mathbb{R}$.
If $u\leq v$ on $\partial\Omega$,
then $u\leq v$ in $\Omega$.
\end{lemma}

\begin{theorem}[\textbf{Comparison principle -} {\cite[Theorem~2.5]{LinLiu2022}}]
\label{thm:Liu-comparison}
Suppose $f(x)\in C(\Omega)\cap L^\infty(\Omega)$ with $f(x)\leq 0$,
$\lambda(x)\in C(\Omega)\cap L^\infty(\Omega)$ with
$\inf_{\Omega}\lambda(x)>0$ and $m\in(0,3)$. Let $u$ and
$v\in C(\overline{\Omega})$ satisfying
\[
\Delta_\infty u(x)-\lambda(x)(u(x))_+^m\geq f(x),
\qquad x\in\Omega,
\]
and
\[
\Delta_\infty v(x)-\lambda(x)(v(x))_+^m\leq f(x),
\qquad x\in\Omega
\]
in the viscosity sense. If $v\geq u$ on $\partial\Omega$ and
$v\geq 0$ on $\partial\Omega$, then $v\geq u$ in $\Omega$.
\end{theorem}

\begin{theorem}[\textbf{A priori $L^\infty$ bounds-} {\cite[Theorem 5.3]{BhattacharyaMohammed2012}}]
\label{thm:L infty}
Let $\Omega\subset\mathbb{R}^n$ be a bounded domain, $g\in C(\partial\Omega)$, and
$f\in C(\Omega\times\mathbb{R},\mathbb{R})$ satisfies that for every compact interval $I\subseteq R$,
\begin{equation}
\label{local bdd f}
\sup_{\Omega\times I}|f(x,t)|<\infty.
\end{equation}
Assume further that
\begin{equation}
\label{growth condition f}
\begin{cases}
\displaystyle
\liminf_{t\to+\infty}
\frac{\inf_{x\in\Omega}f(x,t)}{t^3}
=\mathfrak L_{-},
\\[1.2ex]
\displaystyle
\liminf_{t\to-\infty}
\frac{\sup_{x\in\Omega}f(x,t)}{t^3}
=\mathfrak L_{+},
\end{cases}
\end{equation}
for some $\mathfrak L_{\pm}\in[0,\infty]$. Then there exists a constant
$C>0$, depending only on $f$, $g$, and $\operatorname{diam}(\Omega)$, such that
\[
\|u\|_{L^\infty(\Omega)}\le C
\]
for every viscosity solution $u\in C(\overline{\Omega})$ of
\eqref{eq:general-dirichlet}.
\end{theorem}

\begin{theorem}[\textbf{Existence -} {\cite[Theorem 5.5]{BhattacharyaMohammed2012}}]
\label{thm:existence}
Let $\Omega\subset\mathbb{R}^n$ be a bounded domain and $g\in C(\partial\Omega)$.
If $f\in C(\Omega\times\mathbb{R},\mathbb{R})$ satisfies
\eqref{local bdd f} and \eqref{growth condition f}, then the
Dirichlet problem \eqref{eq:general-dirichlet} admits a viscosity solution
$u\in C(\overline{\Omega})$.
\end{theorem}

\begin{theorem}[\textbf{Existence and non-negativity}]\label{prop:existence}
Let $B_1\subset\mathbb{R}^n$ be a bounded domain. Assume that$\alpha>0$, $0<m<3$, $h\in C(\Omega)\cap L^\infty(\Omega)$ and $h\le0$, $g\in C(\partial\Omega)$ and $g\ge0$.
Then there exists at least one viscosity solution to the Dirichlet problem
\begin{equation*}
\begin{cases}
    \Delta_\infty u=|x|^\alpha u_+^m+h & \text{in }B_1,\\
    u=g & \text{on }\partial B_1
\end{cases}
\end{equation*}
Moreover, every viscosity solution of the Dirichlet problem is bounded in $L^{\infty}(B_1)$ and nonnegative.
\end{theorem}

Actually, the existence and the priori $L^{\infty}$-estimate for our problem can be viewed as a direct corollary of Theorem \ref{thm:L infty} and \ref{thm:existence}.

\begin{proof}
Set
\[
F(x,t):=|x|^\alpha(t_+)^m+h(x),
\qquad
(x,t)\in B_1\times\mathbb{R}.
\]

Clearly, for any compact interval $I\subset\mathbb R$,
\[
\sup_{(x,t)\in B_1\times I}|F(x,t)|<\infty.
\]

Then we verify the condition \eqref{local bdd f}.

For $t>0$, since $m<3$ and $|x|^\alpha\leq1$ in $B_1$, by direct computation, we yield
\[
\sup_{x\in B_1}
\left|
\frac{F(x,t)}{t^3}
\right|
\leq
t^{m-3}
+
\frac{\|h\|_{L^\infty(B_1)}}{t^3}
\to0
\qquad
\text{as }t\to+\infty,
\]
thus
\[
\liminf_{t\to+\infty}\inf_{x\in B_1}\frac{F(x,t)}{t^3}=0.
\]

On the other hand, for $t<0$, $F(x,t)=h(x)$.

We have
\[
\sup_{x\in B_1}
\left|
\frac{F(x,t)}{t^3}
\right|
\leq
\frac{\|h\|_{L^\infty(B_1)}}{|t|^3}
\to0
\qquad
\text{as }t\to-\infty,
\]
hence
\[
\liminf_{t\to-\infty}
\sup_{x\in B_1}
\frac{F(x,t)}{t^3}
=0.
\]

Therefore, $\mathfrak L_{\pm}=0$, the conditions fitting into Theorem \ref{thm:L infty}, the existence and a prior $L^{\infty}$-estimate of \eqref{growth condition f} is acquired.

Next we want to show that the viscosity solution is nonnegative. The classical argument is quite similar to \cite{AraujoLeitaoTeixeira2016} and \cite{LinLiu2022}. 

By contradiction, suppose the open set $\mathcal P(u):=\{x\in B_1:u(x)<0\}$ is not empty. Then we yield
\begin{equation*}
\begin{cases}
\Delta_\infty u=h(x)\le0,&\text{in }\mathcal P(u),\\
u=0,&\text{on }\partial\mathcal P(u),
\end{cases}
\end{equation*}
in the viscosity sense. By the comparison principle for infinity-harmonic functions (see \cite{Jensen1993}), $u(x)\ge0$ holds in $\mathcal P(u)$, which leads to a contradiction to the definition of $\mathcal P(u)$.

Thus we proved the existence, $L^{\infty}$-bounds and non-negativity of the solution to \eqref{eq:general-dirichlet}.
\end{proof}

\begin{theorem}[\textbf{Local Lipschitz regularity-}
{\cite[Corollary~2]{Lindgren2014}}]
\label{thm:local-Lipschitz}
Let $u\in C(B_1)$ be a viscosity solution of
\[
\Delta_\infty u=f
\qquad\text{in }B_1,
\]
where $f\in C(B_1)\cap L^\infty(B_1)$. Then $u$ is locally Lipschitz
continuous in $B_1$. In particular,
\[
\|u\|_{C^{0,1}(B_{1/2})}
\leq
C\left(
\|u\|_{L^\infty(B_1)}
+
\|f\|_{L^\infty(B_1)}^{\frac{1}{3}}
\right),
\]
where $C>0$ is a universal constant.
\end{theorem}

\begin{corollary}[\textbf{Local Lipschitz regularity on general balls}]
\label{thm:scaled-local-Lipschitz}
Let $u\in C(B_{2r}(x_0))$ be a viscosity solution of
\[
\Delta_\infty u=f
\qquad\text{in }B_{2r}(x_0),
\]
where $f\in C(B_{2r}(x_0))\cap L^\infty(B_{2r}(x_0))$. Then 
\[
[u]_{C^{0,1}(B_r(x_0))}
\leq
C\left(
\frac{\|u\|_{L^\infty(B_{2r}(x_0))}}{r}
+
r^{1/3}
\|f\|_{L^\infty(B_{2r}(x_0))}^{1/3}
\right),
\]
where $C>0$ is a universal constant and
\[
[u]_{C^{0,1}(B_r(x_0))}
:=
\sup_{\substack{x,y\in B_r(x_0)\\x\neq y}}
\frac{|u(x)-u(y)|}{|x-y|}.
\]
\end{corollary}

\begin{proof}
We do the scaling by letting $v(z)=u(x_0+2rz)$ and direct computation yields
$\Delta_\infty v(z) =(2r)^4f(x_0+2rz)$.

Applying Theorem \ref{thm:local-Lipschitz} to $v$ in $B_1(0)$, we have
\[
[v]_{C^{0,1}(B_{1/2})}
\leq
C\left(
\|v\|_{L^\infty(B_1)}
+
\|(2r)^4f(x_0+2r\,\cdot)\|_{L^\infty(B_1)}^{1/3}
\right).
\]

Moreover, 
$
[v]_{C^{0,1}(B_{1/2})}
=
2r\,[u]_{C^{0,1}(B_r(x_0))}.
$
Absorbing the factors into the universal constant, we yield
\[
[u]_{C^{0,1}(B_r(x_0))}
\leq
C\left(
\frac{\|u\|_{L^\infty(B_{2r}(x_0))}}{r}
+
r^{1/3}
\|f\|_{L^\infty(B_{2r}(x_0))}^{1/3}
\right).
\]
\end{proof}

\begin{lemma}[\textbf{Stability of viscosity solutions- }{\cite[Proposition~4.8]{Koike2014}}] 
\label{lem:stability} 
Let $\Omega\subset\mathbb{R}^{n}$ be open, and let $ F_k,F:\Omega\times\mathbb{R}\times\mathbb{R}^{n}\times\mathbb{S}^{n} \longrightarrow\mathbb{R} $ be continuous functions such that $F_k\to F$ locally uniformly. Suppose that $u_k\in C(\Omega)$ is a viscosity solution of \[ F_k(x,u_k,Du_k,D^2u_k)=0 \qquad\text{in }\Omega, \] and that $u_k\to u$ locally uniformly in $\Omega$. Then $u$ is a viscosity solution of \[ F(x,u,Du,D^2u)=0 \qquad\text{in }\Omega. \] 
\end{lemma}

\begin{theorem}[\textbf{Harnack inequality- }{\cite[Theorem~7.1]{BhattacharyaMohammed2012}}]\label{thm:harnack-source}
Let \(F\in C(\Omega)\cap L^\infty(\Omega)\), and let \(w\in C(\Omega)\) be a nonnegative viscosity solution of
\[
    \Delta_\infty w=F(x)
    \quad\text{in }\Omega.
\]
If \(z\in\Omega\) and \(B_{2r}(z)\subseteq\Omega\), then
\[
    \sup_{B_{2r/3}(z)}w
    \le
    9\inf_{B_{2r/3}(z)}w
    +
    12\cdot \frac{3^{4/3}}{4}
    \left(
        r^4\sup_{B_{2r}(z)}F^+
    \right)^{1/3},
\]
where $F^+=\max\{F,0\}$.

\end{theorem}

\section{Degeneracy-set Comparison and Slice Uniqueness}

One of the main structural differences between our problem and the strong-absorption model studied in \cite{AraujoLeitaoTeixeira2016} is that our weight function is allowed to vanish, and we focus on the situation that the degeneracy set intersects the free boundary. Since the H\'enon-type weight is allowed to vanish at the origin, the available global comparison principles for strong-absorption problems established under a strict positivity assumption on the weight function are no longer directly applicable. We emphasize that this does not imply non-uniqueness. Rather, for the degenerate Dirichlet problem considered here, global uniqueness remains unresolved: at present, we have neither proved it nor found a counterexample.

So we may wonder whether we can get a ``weaker version'' of comparison principle. A useful way to view this issue is to puncture the domain around the degeneracy point. On $\Omega_\varepsilon:=\Omega\setminus\overline{B_\varepsilon(0)}$,
the coefficient satisfies $|x|^\alpha\ge \varepsilon^\alpha>0$,
so the usual comparison argument becomes available. As $\varepsilon\to0$, the artificial inner boundary shrinks to the degeneracy point, suggesting that once the outer boundary condition is given, any possible differences between two solutions may have something to do with the degeneracy point.

The purpose of this section is to make this heuristic  rigorous. We establish a comparison principle showing that the discrepancy between two solutions with the same boundary datum is completely controlled by their discrepancy on $F$. For general case, we can classify the potentially multiple solutions according to the restriction $u|_F$.

Since in the present H\'enon-type model $F=\{0\}$, the family of solutions can consequently be decomposed into slices according to the single value $u(0)$, and uniqueness can be recovered within each such slice.

Let $\mathcal{S}$ denote the set of all nonnegative viscosity solutions of the Dirichlet problem. Since we cannot claim the global uniqueness yet, $\mathcal{S}$ may contain more than one element. For each $c\ge0$, we can define the ``slice'' by
\begin{equation*}
    \mathcal S_c:=\{u\in\mathcal S:u(0)=c\}.
\end{equation*}
Then 
\begin{equation*}
    \mathcal S=\bigcup_{c\geq0}\mathcal S_c,
\end{equation*}
where the slices are disjoint. 

At this stage, the decomposition is purely set-theoretic. A given $\mathcal{S}_c$ could surely contain several distinct solutions. But the results of this section shows that this case actually cannot happen, giving a very elegant structure of the Dirichlet problem, which is
\begin{equation*}
    \#\mathcal S_c\leq1 \qquad\text{for every }c\geq0.
\end{equation*}

In fact, we can prove an even stronger proposition, which is exactly the theorem of degenerate set comparison principle.

\begin{proof}[Proof of Theorem \ref{thm:de-set-comparison}]
The idea of the proof is quite clear and simple. We confine the possible degeneracy of the weight function to a neighborhood of $\partial\Omega\cup F$, where the desired estimate follows from the uniform continuity of $u-v$. Away from this neighborhood, the weight function is strictly positive, and the comparison principle for that can be applied.

Since $u-v\leq 0$ on $\partial\Omega\cup F$ and $u-v\in C(\overline{\Omega})$, $u-v$ is uniformly continuous in $\overline{\Omega}$.

For every $\varepsilon>0$, there exists $\delta>0$ such that $u(x)-v(x)\leq \varepsilon$ whenever $\text{dist}(x,\partial\Omega\cup F)<\delta$.

Set $\Omega_{\delta/2}:=\{x\in \Omega:\text{dist}(x,\partial \Omega\cup F)>\delta/2\}$. Obviously, $\Omega_{\delta/2}\Subset\Omega\setminus F$ and $\inf_{\overline{\Omega_{\delta/2}}} \lambda>0$.

Direct calculation yields
\begin{align*}
    \Delta_\infty(v+\varepsilon)-\lambda(x)(v+\varepsilon)^m&=\Delta_\infty v-\lambda(x)(v+\varepsilon)^m\\
    &\leq h(x)+\lambda(x)v^m-\lambda(x)(v+\varepsilon)^m\\ &\leq h(x)
\end{align*}
in $\Omega_{\delta/2}$ in the viscosity sense.
    
Since $u$ is a viscosity subsolution of the equation,
\begin{equation*}
    \Delta_\infty u(x)-\lambda(x)(u(x))_+^m\ge h(x).
\end{equation*}

Note that on $\partial\Omega_{\delta/2}$, $u\leq v+\varepsilon$ holds. Since $v+\varepsilon\ge u$ on $\partial \Omega_{\delta/2}$ and $v+\varepsilon\ge 0$ on $\Omega_{\delta/2}$, using Comparison Principle with strictly positive weight function (Theorem \ref{thm:Liu-comparison}) in the region of $\Omega_{\delta/2}$, we have 
\begin{equation*}
    u\leq v+\varepsilon \qquad \text{in } \Omega_{\delta/2}.
\end{equation*}

Hence, we have $u-v\leq \varepsilon$ in the whole $\Omega$. Thus we complete the proof.
\end{proof}

The comparison principle on degeneracy set gives an immediate relation between two solutions of the equation. We write down the statement as a corollary.

\begin{corollary}[\textbf{Comparison between two solutions}]
\label{cor:comparison-sol}
Under the assumption of Theorem \ref{thm:de-set-comparison}, let $u,v \in C^0(\Omega)$ be non-negative viscosity solutions of 
\[
\begin{cases}
\Delta_\infty w
=
\lambda(x)w_+^m+h(x)
&\text{in }\Omega,\\
w=g
&\text{on }\partial\Omega.
\end{cases}
\]
Then we have
\begin{equation*}
\sup_{x\in\Omega}|u(x)-v(x)|=\sup_{x\in F}|u(x)-v(x)|.
\end{equation*}
\end{corollary}

\begin{proof}
Denote $M:=\sup_{x\in F}|u(x)-v(x)|$, thus $u\leq v+M$ on $\partial\Omega\cup F$.
Moreover,
\[
u\le v+M
\qquad\text{on }\partial\Omega\cup F.
\]
Hence, by Theorem \ref{thm:de-set-comparison}, $u-v\le M$ in $\Omega$.

Exchanging $u$ and $v$ yields $v-u \le M$ in $\Omega$.

Therefore,
\[
\sup_{x\in\Omega}|u(x)-v(x)|\leq \sup_{x\in F}|u(x)-v(x)|.
\]
The reverse inequality is obvious.
\end{proof}

\begin{remark}
    Here we demonstrate the version that the viscosity solutions $u$ and $v$ have the same boundary value. The case of different boundary value can be dealt with in the same way. We can obtain
    \begin{equation*}
    \sup_{\overline{\Omega}} |u-v|=\max\left\{\sup_{\partial\Omega}|u-v|,\sup_F |u-v|\right\}.
    \end{equation*}
    
    As we mentioned before, the difference of two solutions lies exactly in their value on the boundary and in the degeneracy set of weight function $F$.
\end{remark}

The ``slice uniqueness property'' is a natural corollary of the theorem mentioned above: the two solutions agree on the degenerate set. 

\begin{corollary}[\textbf{Slice Uniqueness}]
\label{cor:slice-uniqueness}
Under the assumptions of Theorem \ref{thm:de-set-comparison} and Corollary \ref{cor:comparison-sol}, if $u=v$ on $F$, then $u=v$ on $\Omega$.
\end{corollary}

In the following corollaries we consider the prototype equation.

\begin{corollary}[\textbf{Symmetry inheritance}]
\label{prop:symmetry-inheritance}
Let $\Omega\subset\mathbb{R}^n$ be a bounded domain containing the origin,
and let $G$ be a subgroup if the orthogonal group $O(n):=\{Q\in\mathbb{R}^{n\times n}:Q^TQ=I\}$ such that $Q\Omega=\Omega$ for every $Q\in G$.

Assume that the boundary datum $g$ and the source term $h$ is $G$-invariant, namely,
$g(Qx)=g(x)$ on $\partial\Omega$ and $h(Qx)=h(x)$. Then every viscosity solution of \eqref{eq:general-dirichlet} inherits the same symmetry, that is,
$u(Qx)=u(x)$ for every $Q\in G$.
\end{corollary}

\begin{proof}
Fix $Q\in G$ and set $u_Q(x):=u(Qx)$. Since $Q$ is orthogonal,
$|Qx|=|x|$, and the infinity Laplacian is invariant under orthogonal
transformations. Hence $u_Q$ solves the same equation as $u$ in $\Omega$.

Moreover, $u_Q=g$ on $\partial\Omega$ since $g$ is $G$-invariant, and
$u_Q(0)=u(0)$ since $Q0=0$. Thus, $u_Q$ and $u$ belong to the same slice.
By Corollary \ref{cor:slice-uniqueness}, we conclude that $u_Q\equiv u$.
\end{proof}

\begin{corollary}[\textbf{Radial symmetry}]
\label{cor:radial-symmetry}
Let $\Omega=B_R$ and suppose that $g\equiv c$ on $\partial B_R$. Assume $h\in C(B_R)\cap L^\infty(B_R)$, $h\le0$,
is radially symmetric, that is, $h(Qx)=h(x)$.
Then every viscosity solution is radially symmetric.
\end{corollary}

\begin{proof}
Apply Corollary \ref{prop:symmetry-inheritance} with $G=O(n)$.
\end{proof}

\begin{corollary}[\textbf{Dead-core geometry for the zero slice}]
\label{cor:radial-dead-core}
Under the assumptions of Corollary \ref{cor:radial-symmetry}, suppose
$u\ge0$ and $u(0)=0$. Denote
\[
    \mathcal{D}(u):=\operatorname{int}\{u=0\}.
\]
Then the following cases hold.

If $0\in\mathcal{D}(u)$, then the connected component of $\mathcal{D}(u)$
containing the origin is a ball centered at the origin.

If instead $0\in\partial\{u>0\}$, then no ball of positive radius centered
at the origin can be contained in $\{u=0\}$.
\end{corollary}

\begin{proof}
By Corollary \ref{cor:radial-symmetry}, $u(x)=U(|x|)$, and hence
$\{u=0\}$ and $\mathcal D(u)$ are radially symmetric.

If $0\in\mathcal D(u)$, the connected component of $\mathcal D(u)$ containing the origin is therefore a ball $B_\rho$ for some $\rho>0$. 
On the other hand, if $0\in\partial\{u>0\}$, every neighborhood of the origin contains points of $\{u>0\}$, and hence no $B_\rho$, $\rho>0$ can be contained in $\{u=0\}$.
\end{proof}

\begin{remark}
\label{rem:radial-dead-core}
Corollary \ref{cor:radial-dead-core} does not give a complete classification of the radial zero set. In particular, $0\in\partial\{u>0\}$ only excludes a dead-core ball of positive radius centered at the origin, and does not by itself imply that $\{u=0\}=\{0\}$.

For the homogeneous H\'enon-type equation, a smooth radial profile satisfies $(U')^2U''=r^\alpha U^m$ in its positive phase, and hence $\left( (U')^3 \right)'=3r^\alpha U^m>0$. This suggests that once $U$ enter the positive phase, it should not return to the zero phase.

\end{remark}

\section{Higher regularity estimates: Proof of Theorem \ref{thm:regularity}}
In this section we will prove Theorem \ref{thm:regularity} by means of a triadic iteration argument.

\begin{proof}[Proof of Theorem \ref{thm:regularity}]

The proof is quite similar to the homogeneous situation in \cite{BezerraDaSilvaNascimentoSa2024}. We do the discrete iterative techniques with continuous reasoning likewise, and put our attention to the inhomogeneous term $h(x)$. It suffices to claim the existence of a universal constant $C_0>0$, such that for all $j\in\mathbb{N}$, we have
\begin{equation}
\label{regularity C0}
S_{j+1}
\le
\max\left\{
C_0\,3^{-\beta(j+1)},
\,3^{-\beta}S_j
\right\},
\end{equation}
where 
\[
S_j:= \sup_{B_{3^{-j}}}u\quad\text{and}\quad
\beta:=\frac{4+\alpha}{3-m}.
\]

We argue by contradiction. Suppose that \ref{regularity C0} fails to hold, which is for each $k\in\mathbb{N}$ there exists $j_k\in\mathbb{N}$ such that
\begin{equation}
\label{regularity contradiction}
S_{j_k+1}
>
\max\left\{
k\,3^{-\beta(j_k+1)},
\,3^{-\beta}S_{j_k}
\right\}.
\end{equation}

For each $k\in\mathbb{N}$, we define the rescaled function
$v_k:B_1\to\mathbb{R}$ by
\[
    v_k(x):=
    \frac{u(3^{-j_k}x)}{S_{j_k+1}}.
\]

By the definition of the rescaled function, we yield
\begin{align}
\label{rescaled relation 1}
0\le v_k(x)
\le
\frac{S_{j_k}}{S_{j_k+1}}
<
3^\beta;\\
\label{rescaled relation 2}
v_k(0)=0;\\
\label{rescaled relation 3}
\sup_{B_{1/3}}v_k
=
\frac{
\sup_{B_{3^{-(j_k+1)}}}u
}{
S_{j_k+1}
}
=1.
\end{align}

Through direct computation,
\begin{align}
\Delta_\infty v_k
&=
\left\langle
D^2v_k(x)Dv_k(x),Dv_k(x)
\right\rangle
\nonumber\\
&=
\left\langle
\frac{3^{-2j_k}}{S_{j_k+1}}
D^2u\!\left(3^{-j_k}x\right)
\left(
\frac{3^{-j_k}}{S_{j_k+1}}
Du\!\left(3^{-j_k}x\right)
\right),
\left(
\frac{3^{-j_k}}{S_{j_k+1}}
Du\!\left(3^{-j_k}x\right)
\right)
\right\rangle
\nonumber\\
&=
\frac{3^{-4j_k}}
{S_{j_k+1}^{\,3}}
\Delta_\infty u\!\left(3^{-j_k}x\right)
\nonumber\\
&=:F_k,
\label{eq:scaled-infinity-laplacian}
\end{align}
here $F_k$ satisfies
\begin{equation*}
F_k(x)
=
\frac{3^{-(4+\alpha)j_k}}
{S_{j_k+1}^{3-m}}
|x|^\alpha v_k^m(x)
+
\frac{3^{-4j_k}}
{S_{j_k+1}^{3}}
h(3^{-j_k}x).
\end{equation*}

Note that $h\leq 0$, we analysis that
\begin{equation*}
F_k^+(x)
\le
\frac{3^{-(4+\alpha)j_k}}
{S_{j_k+1}^{3-m}}
|x|^\alpha v_k^m(x).
\end{equation*}

Using \eqref{regularity contradiction}, we have $S_{j_k+1}
>
k\,3^{-\beta(j_k+1)}$ and
$S_{j_k+1}
>
3^{-\beta}S_{j_k}
$, combined with $|x|^{\alpha}\leq1$ inside $B_1$, thus we can control the positive part of $F_k$.
\begin{equation*}
\|F_k^+\|_{L^\infty(B_1)}
\le
\frac{3^{4+\alpha+\beta m}}{k^{3-m}}.
\end{equation*}

Then we are going to employ the Harnack inequality \ref{thm:harnack-source}, together with (\ref{rescaled relation 2}), (\ref{rescaled relation 3}) and (\ref{eq:scaled-infinity-laplacian}), take $r=1/2$ and
\begin{align}
\nonumber
1&=\sup_{B_{1/3}}v_k\\
\nonumber
&\le
9\inf_{B_{1/3}}v_k
+
12\cdot\frac{3^{4/3}}{4}\cdot\left(\frac{1}{2}\right)^{4/3}\left(\sup_{B_{1}} F_k^+\right)^{1/3}\\
\nonumber
&\le
12\cdot\frac{3^{4/3}}{4}\cdot\left(\frac{1}{2}\right)^{4/3}\left(\sup_{B_{1}} F_k^+\right)^{1/3}\\
\nonumber
&\to0\quad \mathrm{as}\quad k\to\infty,
\end{align}
which leads to a contradiction. Thus we can claim that we have such a constant $C_0$ that (\ref{regularity C0}) holds.

By mathematical induction, we have $S_j\leq \max\{C_0,S_0\}\cdot 3^{-\beta j}$ for $j\in\mathbb{N}$. 

Finally, given $r\in(0,\frac{1}{2})$, let $j\in\mathbb{N}$ be such that 
\[
3^{-(j+1)}\leq r\leq 3^{-j}.
\]
Then we have
\[
\sup_{B_r}u \le \sup_{B_{3^{-j}}}u
=
S_j 
\le M3^{-\beta j}
\le 3^{\beta}\max\{C_0,S_0\}\cdot r^{\beta}.
\]

We denote $C:=3^\beta\max\{C_0,S_0\}$, thus the theorem is complete.
\end{proof}

\begin{remark}
As we mentioned before, here we only consider the special H\'enon-type model in which, apart from the inhomogeneous term, the right hand side is exactly $|x|^{\alpha}u_+^m$ and the degeneracy point is fixed at $x_0=0$.

Here we employ triadic iteration to ensure it fits into the condition that Harnack inequality requires. Moreover, the feature of the Harnack inequality (Theorem \ref{thm:harnack-source}) is that only the positive part enters the estimate. Thus a nonpositive source term \(h\le0\) does not affect the regularity.

For consistency with the non-degeneracy result established later, we state the regularity estimate under the assumption $0\in\partial\{u>0\}$. However, this condition can be weakened to $u(0)=0$. In fact, for an nonnegative continuous solution, $u(0)=0$ includes either an interior point of the dead core or a free boundary point. It is worth pointing out that the regularity estimate only requires $u(0)=0$, whereas the for non-degeneracy estimate the free boundary assumption is much more essential.
\end{remark}

\section{Non-degeneracy estimates}

In this section, we will prove Theorem \ref{thm:non-de-homo} and Theorem \ref{thm:non-de-inhomo}. 

We first consider the homogeneous case then we handle the inhomogeneous case with a nonpositive term $h$.

\subsection{The homogeneous case}

For the homogeneous case, we give two proofs. The first one follows the penalization approximation used in the existing literature while the second one is based directly on the degeneracy-set comparison principle established in the previous section. The essence of both methods is the degeneracy-set comparison principle.

Recall the slice uniqueness in the last section, under current assumption of $F=\{0\}$, two nonnegative solutions with same boundary datum coincide provided they agree at the origin. This observation will be used below to identify the limit of the penalized family with the prescribed solution $u$.

The use of penalized problems is closely related to, but does not fall within the limiting solution framework adopted in \cite{WangJiang2026} and \cite{BezerraDaSilvaNascimentoSa2024}. In Wang and Jiang's work, the solution is assumed a priori to arise as a locally uniform limit of the penalized problem family. In the present setting, however, the slice uniqueness established in the last section allows us to recover the penalized limit to the prescribed solution from the equation itself. To be more precise, if $u_k$ denotes the corresponding penalized family and $u_k\to u_\infty$, combined with same value at the degeneracy point as the prescribed solution $u$, slice uniqueness yields that $u_\infty=u$. Therefore, in our work, although $\{u_k\}$ does not constitute a limiting approximation in the sense of Wang and Jiang, slice uniqueness identifies its limit with the prescribed solution without imposing an additional assumption on $u$.

This recovery mechanism also suggests that the limiting solution property assumption in the critical point non-degeneracy estimate of \cite{WangJiang2026} might be removable, provided a similar degeneracy set comparison principle and slice uniqueness property can be obtained in their more general setting. 

Although only first proof of the homogeneous case will use the penalized approximation, we establish it in the more general inhomogeneous setting ($h\le0$), since this requires only a minor additional argument and the result is of independent interest. We state the results as a proposition. 

\begin{proposition}[\textbf{Penalized approximation}]
\label{prop:penalized-approximation}
Let $\alpha>0$, $0<m<3$, and $g\in C(\partial B_1)$, $g\ge0$. Let $u\in C(\overline{B_1})$ be a nonnegative viscosity solution to \eqref{eq:general-dirichlet}, here we suppose $u(0)=0$.

For each $k$, consider the penalized Dirichlet problem
\[
\begin{cases}
\Delta_\infty u_k
=
|x|^\alpha (u_k)_+^m+h(x)+\dfrac1k
&\text{in }B_1,\\[2mm]
u_k=u=g
&\text{on }\partial B_1.
\end{cases}
\tag{$P_k$}
\]

Then, for each $k\in\mathbb{N}$, $(P_k)$ admits at least one viscosity solution $u_k\in C(\overline{B_1})$. Moreover, it holds that $u_1\le u_2\le\cdots\le u_k\le u_{k+1}\le\cdots\le u$ in $\overline{B_1}$.

In particular, $u_k \to u$ uniformly in $\overline{B_1}$.
\end{proposition}

\begin{proof}
We denote the right hand side of $(P_k)$ as
\[
    G_k(x,t):=|x|^\alpha (t_+)^m+h(x)+\frac1k.
\]

Since $0<m<3$, $G_k$ satisfies condition \eqref{local bdd f} and \eqref{growth condition f}, the existence of the solution of the penalized problem $(P_k)$ is guaranteed by Theorem \ref{thm:existence}.

By the comparison principle, it is easy to see that the family $\{u_k\}$ is ordered. Note that
\begin{equation}
\label{eq: penalized-order}
u_1 \leq u_2 \leq \cdots u_k \leq u_{k+1} \leq \cdots \leq u,
\end{equation}

The uniform boundedness of $\{u_k\}$ follows immediately. Indeed, set $M:=\max\left\{\|u_1\|_{L^\infty(B_1)},\|u\|_{L^\infty(B_1)}\right\}$, thus $\|u_k\|_{L^\infty(B_1)} \le M$  for every $k$. Since every bounded monotone sequence converges, we may define the $u_{\infty}$ pointwise for every $x\in\overline{B_1}$ by
\[
u_{\infty}(x):=\lim_{k\to\infty}u_k(x).
\] 

We next show that $u_{\infty}$ coincides with the prescribed solution $u$. It suffices to prove that $u_\infty(0)=u(0)$.

For $\Delta_\infty u_k(x)=|x|^\alpha (u_k)_+^m(x)+h(x)+\frac{1}{k}$, Through Theorem \ref{thm:local-Lipschitz}, we have
\begin{equation*}
\|u_k\|_{C^{0,1}(B_{1/2})}\leq C\left(M+\|G_k(\cdot,u_k(\cdot))\|_{L^{\infty}(B_1)}^{1/3}\right),
\end{equation*}

We claim that $u_{\infty}$ is also a nonnegative viscosity solution to the original equation. We prove the viscosity solution property by the stability theorem and the nonnegativity by giving an auxiliary function.

Here Theorem \ref{lem:stability} guarantees that $u_{\infty}$ is actually a solution of the original Dirichlet problem.

By the uniform $L^\infty$ bound for $\{u_k\}$ and $h$, $\bigl\|G_k(\cdot,u_k(\cdot))\bigr\|_{L^\infty(B_1)}\leq 1\cdot M^m+\|h\|_{L^{\infty}(B_1)}+1:=Q< \infty$.

Take a set $K\Subset B_1$, denote $\delta:=\text{dist}(K,\partial B_1)$. We have $B_{\delta/2}(x)\Subset B_1$ for every $x\in K$.

Hence for all $x,y\in K$, as long as $|x-y|<\delta/4$, we have \[|u_k(x)-u_k(y)|\leq C_K|x-y|,\]
where $C_K:=C_0\left(
\frac{2\|u_k\|_{L^\infty(B_{\delta/2}(x))}}{\delta}
+
\left(\frac{\delta}{2}\right)^{1/3}
\bigl\|G_k(\cdot,u_k(\cdot))\bigr\|_{L^\infty(B_{\delta/2}(x))}^{1/3}
\right)+1$ is independent of both $x$ and $k$, demonstrating that $u_k$ is equicontinuous on $K$.

Moreover, since the inequality holds for every $k$, let $k\to\infty$, we have $u_{\infty}\in C(K)$.

Note that $K$ is a compact set, $u_k$, $u_{\infty}\in C(K)$, $u_k$ converges to $u_{\infty}$ pointwise and $u_k(x)$ is monotonically on $K$ in regard to $k$, therefore Dini's theorem yields $u_k\to u_\infty$ uniformly on $K$.

As $K\Subset B_1$ is arbitrary, $u_k\to u_\infty$ locally uniform in $B_1$.

Since $u_1=u=g$ on $\partial B_1$,
for every $\xi\in\partial B_1$, we have
\[
g(\xi)
=
\lim_{x\to\xi}u_1(x)
\le
\liminf_{x\to\xi}u_\infty(x)
\le
\limsup_{x\to\xi}u_\infty(x)
\le
\lim_{x\to\xi}u(x)
=
g(\xi).
\]

Therefore,
\[
\lim_{x\to\xi}u_\infty(x)=g(\xi),
\qquad \forall\,\xi\in\partial B_1,
\]showing that $u_{\infty}\in C(\overline{B_1})$.

Then we invoke the stability property of viscosity solutions, we denote
$$F_k(x,r,p,X)=\langle Xp,p\rangle-|x|^\alpha(r_+)^m-h(x)-1/k$$ and $$F(x,r,p,X)=\langle Xp,p\rangle-|x|^\alpha(r_+)^m-h(x).$$ 

Hence $F_k$ converges to $F$ locally uniformly since $\sup|F_k-F|=\frac{1}{k}\rightarrow0$ as $k\to \infty$. Since $u_k$ converges to $u_{\infty}$ locally uniformly in $B_1$, thus $u_{\infty}$ satisfies $F=0$. Combined with $u_{\infty}=g$ on $\partial B_1$, it holds that $u_{\infty}$ is also a viscosity solution to the Dirichlet problem.

We next identify $u_{\infty}$ with the origin viscosity solution. The sequence of solutions of $(P_k)$ is actually not all nonnegative, still we can show that $u_{\infty}\ge0$, which applies to Lemma \ref{cor:comparison-sol}.

Consider
\[
w_{k,\delta}(x)
:=
\frac{3^{4/3}}{4}\left(\frac{1}{k}+\delta\right)^{1/3}
\left(|x|^{4/3}-1\right),
\]
here $\delta>0$.

Note that $w_{k,\delta}\leq 0$ in $B_1$, thus $(w_{k,\delta})_+=0$. Also, by direct computation, 
\[
\Delta_{\infty}w_{k,\delta}
=
|x|^{\alpha}(w_{k,\delta})_{+}^m+\frac{1}{k}+\delta
\]
in the viscosity sense.

Also, we have 
\[
\Delta_{\infty}u_k
=
|x|^{\alpha}(u_k)_{+}^m+h(x)+\frac{1}{k}.
\]

Consider the boundary of $\partial B_1$, it holds that
\[
w_{k,\delta}(\xi)=0\leq g(\xi)=u_k(\xi),
\]
for any $\xi$ on $\partial B_1$.

Hence the comparison principle (Lemma \ref{lem:boundary-comparison}) gives that $w_{k,\delta}\leq u_{k}$ inside the whole $\overline{B_1}$ since $h\leq0$.

Thus actually $u_k(x)\ge -\frac{3^{4/3}}{4}(\frac1k+\delta)^{1/3}$ holds. Let $\delta\to0$ and $k\to \infty$, it yields that 
\[
u_{\infty}\ge 0.
\]

Since $u(0)=0$ and $w_{k,\delta}\leq u_k\leq u_{\infty}\leq u$ hold, $u_\infty(0)=0$.

We have therefore reached the pleasant conclusion that the signed penalized approximations $u_k$ converge uniformly to the prescribed solution $u$ using the slice uniqueness property, which is
\[
u_{\infty}=u.
\]

Since $u_k\uparrow u$ pointwise on the compact set $\overline{B_1}$, and $u_k$, $u\in C(\overline{B_1})$, Dini's theorem yields $u_k\to u$ uniformly on $\overline{B_1}$.

Thus we complete the proof.
\end{proof}

Next we enter the proof of the non-degeneracy estimate. 

\begin{proof}[\textbf{Proof I} of Theorem \ref{thm:non-de-homo}]
Take $h\equiv0$. With the identification \ref{prop:penalized-approximation}, we are free to return our attention to the original viscosity solution while still enjoying the convenience of the sequence of $\{u_k\}$ to do the following approximation.

Then we have a sequence $\{u_k\}$ that uniformly converges to $u$. Similar to a standard idea in free boundary analysis, see \cite{Figalli2018}, we take a sequence $\{x_\ell\}$ in the positive phase to approximate the free boundary point.

Fixing $r\in(0,1)$, we have the following argument.

Since $0\in\partial\{u>0\}$, we may choose a sequence $x_\ell\in\{u>0\}\cap B_r$ such that $x_\ell\to0$. For every fixed $\ell$, we have $u_k(x_\ell)\to u(x_\ell)$ as $k\to\infty$. 

For every $\varepsilon>0$, there exists $K_\ell\in\mathbb{N}$ such that $|u_k(x_\ell)-u(x_\ell)|<\varepsilon$ for every $k\geq K_\ell$.
Since $x_\ell\in\{u>0\}$, we have $u(x_\ell)>0$. Taking $\varepsilon=\frac{u(x_\ell)}{2}$, we obtain
$u_k(x_\ell)>\frac{1}{2}u(x_\ell)>0$
for every $k\geq K_\ell$.

We may now choose $\{k_{\ell}\}$ inductively by taking
\[
k_1\ge\max\{1,K_1\},
\]
and
\[
k_\ell
\ge
\max\{\ell,K_\ell,k_{\ell-1}+1\},
\qquad \ell\ge2.
\]
Then we can guarantee that $k_\ell\to\infty$.

By direct computation, we know that the function
\[
\Phi_\ell(x)
:=
\left(
\frac{(3-m)^4}
{(4+\alpha)^3(1+\alpha+m)}
\right)^{\frac1{3-m}}
\bigl(|x|-|x_\ell|\bigr)_+^\beta.
\]
satisfies
\begin{equation*}
\begin{aligned}
\Delta_\infty\Phi_\ell
&=
\bigl(|x|-|x_\ell|\bigr)_+^\alpha
\Phi_\ell^m\\
&\le
|x|^\alpha\Phi_\ell^m\\
&<
|x|^\alpha(\Phi_\ell)_+^m+\frac1{k_\ell}.
\end{aligned} 
\end{equation*}

Also, it holds that $$\Delta_\infty u_{k_{\ell}}(x)=|x|^\alpha (u_{k_{\ell}})_+^m(x)+\frac{1}{k_{\ell}}.$$

We now claim that there exists a point $z_\ell\in\partial B_r(0)$ such that $u_{k_\ell}(z_{\ell})\ge\Phi_{\ell}(z_{\ell})$. If this were not the case, we argue by contradiction. By applying the comparison principle (Lemma \ref{lem:boundary-comparison}), we would conclude that  $u_{k_\ell}\le\Phi_\ell$ in the entire ball $B_r(0)$. However according to the construction above,
\[
0
<
u_{k_\ell}(x_\ell)
\le
\Phi_\ell(x_\ell)
=
0,
\]
which gives a contradiction.

Therefore, we yield the following estimate
\[
\sup_{\partial B_r(0)}u(x)
\ge
\sup_{\partial B_r(0)}u_{k_\ell}
>
\left(
\frac{(3-m)^4}
{(4+\alpha)^3(1+\alpha+m)}
\right)^{\frac1{3-m}}(r-|x_\ell|)^\beta.
\]

Finally, let $\ell\to\infty$, and $x_{\ell}\to0$, we conclude that 
\[
\sup_{\partial B_r(0)}u(x)
\ge
\left(
\frac{(3-m)^4}
{(4+\alpha)^3(1+\alpha+m)}
\right)^{\frac1{3-m}}r^{\frac{4+\alpha}{3-m}}.
\]
Therefore we finish the proof.
\end{proof}

The penalization procedure is not intrinsic for the non-degeneracy estimate. Nevertheless, we retain the preceding argument as it has more connections with the penalization and limiting solution framework used in the existing literature.

In fact, the argument can be carried out directly for the original solution by using the degeneracy set comparison principle Theorem \ref{thm:de-set-comparison} we have already developed, without passing through the penalized family and can be extended to the inhomogeneous case more clearly.

The notations are the same as the proof above.

\begin{proof}[\textbf{Proof II} of Theorem \ref{thm:non-de-homo}]
As shown above,
\[
\Delta_\infty\Phi_{\ell}
\leq
|x|^\alpha(\Phi_{\ell})_+^m
\qquad\text{in }B_r
\]
in the viscosity sense.

Note that
\[
u(0)=0=\Phi_{\ell}(0),
\]
thus $u$ and $\Phi_{\ell}$ agree on the degeneracy set of the weight function $F$.

Now, for any $r\in(0,1)$, we claim that there must exist a point $z_r\in\partial B_r(0)$ such that $u(z_r)\ge \Phi_{\ell}(z_r)$. Otherwise, by degeneracy set comparison principle (Theorem \ref{thm:de-set-comparison}), $\Phi_{\ell}\ge u$ in the whole region $B_r(0)$. However, $0=\Phi_{\ell}(x_{\ell})<u(x_{\ell})$. 

By similar process, it holds that
\[
\sup_{\partial B_r(0)}u(x)
\ge
\left(
\frac{(3-m)^4}
{(4+\alpha)^3(1+\alpha+m)}
\right)^{\frac1{3-m}}r^{\frac{4+\alpha}{3-m}}.
\]

Thus, although the penalized family $\{u_k\}$ provides a convenient approximation framework, the non-degeneracy estimate itself can also be obtained directly from the comparison argument on the degeneracy set.
\end{proof}

\subsection{The inhomogeneous case}

After establishing the sharp non-degeneracy estimate for the homogeneous case, we now focus on the case where a source term $h$ is considered. Inspired by the framework and techniques in Theorem 3.2 of \cite{LinLiu2022}, We first prove a weaker form of non-degeneracy of the viscosity solutions at positive-phase point where the source term vanishes. Then we try to approximate the very point $0$ using the powerful degeneracy set comparison principle Theorem \ref{thm:de-set-comparison} we have developed.

\begin{proof}[Proof of Theorem \ref{thm:non-de-inhomo}]
We divided the proof into two steps.

\textbf{Step 1.} First, since $0\in \overline{\{u>0\}\cap\{h=0\}}$, there exists $x_k\in\{u>0\}\cap\{h=0\}$ such that $x_k\to0$.

Consider $B_{\theta|x_k|}(x_k)\Subset B_1$. For any $x\in B_{\theta|x_k|}(x_k)$, define
\begin{equation}
\label{non-de-h-radical}
V_k(x)
:=
s|x_k|^{\frac{\alpha}{3-m}}
|x-x_k|^{\frac4{3-m}}.
\end{equation}

Direct computation yields
\begin{equation*}
\begin{aligned}
\Delta_\infty V_k-|x|^\alpha V_k^m
\le&
-s^m
\left[
(1-\theta)^\alpha
-
\frac{64(1+m)}{(3-m)^4}s^{3-m}
\right]
\cdot
|x_k|^{\frac{3\alpha}{3-m}}
|x-x_k|^{\frac{4m}{3-m}}
\\
\le&\ h(x)
\end{aligned}    
\end{equation*}
in $B_{\theta|x_k|}(x_k)\setminus\{x_k\}$.

Easy to verify the supersolution inequality holds at $x_k$ in the viscosity sense.

We claim that there exists $z_k\in\partial B_{\theta|x_k|}(x_k)$ such that $u(z_k)>V_k(z_k)$. The argument here is similar to Theorem 3.2 in \cite{LinLiu2022}, so we omit it here.

Hence we have 
\begin{equation}
\label{step1-de-h-1}
u(z_k)>V_k(z_k)=
s\theta^{\frac4{3-m}}
|x_k|^{\frac{4+\alpha}{3-m}}.
\end{equation}

Also, we have $|z_k|\le (1+\theta)|x_k|$, thus 
\begin{equation*}
u(z_k)
>
\frac{
s\theta^{\frac4{3-m}}
}{
(1+\theta)^{\frac{4+\alpha}{3-m}}
}
|z_k|^{\frac{4+\alpha}{3-m}}.
\end{equation*}

\textbf{Step 2.} Set
\[
C_{\alpha,m}:=\left(\frac{(3-m)^4}
{(4+\alpha)^3(1+\alpha+m)}
\right)^{\frac1{3-m}}.
\]

For fixed $\eta\in(0,1)$, we define 
\begin{equation*}
\Phi_k(x)
=
\frac{
s\theta^{\frac4{3-m}}
}{
(1+\theta)^{\frac{4+\alpha}{3-m}}
}
|z_k|^{\frac{4+\alpha}{3-m}}
+
\eta C_{\alpha,m}
\bigl(|x|-|z_k|\bigr)_+^{\frac{4+\alpha}{3-m}}.
\end{equation*}

Denote $b:=\dfrac{
s\theta^{\frac4{3-m}}
}{
(1+\theta)^{\frac{4+\alpha}{3-m}}
}$. Notice immediately that
\begin{equation*}
\Phi_k(z_k)
=
b|z_k|^{\frac{4+\alpha}{3-m}}
<
u(z_k).
\end{equation*}

We have $\Phi_k$ as a supersolution. In fact, for $|x|\le |z_k|$,
\begin{equation*}
\begin{aligned}
\Delta_\infty\Phi_k
-|x|^\alpha\Phi_k^m
&=
-\left(
\frac{
s\theta^{\frac4{3-m}}
}{
(1+\theta)^{\frac{4+\alpha}{3-m}}
}
\right)^m
|x|^\alpha
|z_k|^{\frac{m(4+\alpha)}{3-m}}
\\
&\le
-\left(
\frac{
s\theta^{\frac4{3-m}}
}{
(1+\theta)^{\frac{4+\alpha}{3-m}}
}
\right)^m
|x|^{\frac{3\alpha+4m}{3-m}}.
\end{aligned}
\end{equation*}

And for $|x|>|z_k|$, denote $t:=\frac{|z_k|}{|x|}\in(0,1)$, and direct computation yields
\begin{equation*}
\begin{aligned}
&\Delta_\infty\Phi_k
-|x|^\alpha\Phi_k^m
\\
={}&
-|x|^{\frac{3\alpha+4m}{3-m}}
\Bigg\{
\left[
bt^{\frac{4+\alpha}{3-m}}
+
\eta C_{\alpha,m}
(1-t)^{\frac{4+\alpha}{3-m}}
\right]^m
-
\eta^3C_{\alpha,m}^m
(1-t)^{\frac{3\alpha+4m}{3-m}}
\Bigg\}.
\end{aligned}
\end{equation*}

Set $G_\eta(t) 
:=
\left[
bt^{\frac{4+\alpha}{3-m}}
+
\eta C_{\alpha,m}
(1-t)^{\frac{4+\alpha}{3-m}}
\right]^m
-
\eta^3C_{\alpha,m}^m
(1-t)^{\frac{3\alpha+4m}{3-m}}.    
$

Then for $0\le t<1$, 
\begin{equation*}
G_\eta
\ge
\eta^mC_{\alpha,m}^m
(1-t)^{\frac{m(4+\alpha)}{3-m}}
\left[
1-\eta^{3-m}(1-t)^\alpha
\right]
>0.
\end{equation*}

When $t=1$, $G_\eta(1)=b^m>0$. Hence we can select a constant $\mu_0(\eta):=\min_{0\le t \le1}G_\eta(t)>0$, such that
\begin{equation*}
\Delta_\infty\Phi_k
-|x|^\alpha\Phi_k^m
\le
-G_\eta(t)
|x|^{\frac{3\alpha+4m}{3-m}}
\le
-\mu_0(\eta)
|x|^{\frac{3\alpha+4m}{3-m}}.
\end{equation*}

According to the condition \eqref{cond:non-de-h2}, $\Phi_k$ is a supersolution, that is
\begin{equation*}
\Delta_\infty\Phi_k-|x|^\alpha\Phi_k^m
\le h(x)
\qquad\text{in }B_{r_\eta}.
\end{equation*}

For any given $r\in(0,r_{\eta})$, since $z_k\to0$, we can choose sufficiently large $k$ and $|z_k|<r$.

Again, we claim that there exists $z_r$ on $\partial B_r$ such that $u(z_r)>\Phi_k(z_r)$.

We argue that by contradiction. Since it holds that 
$u(0)=0
<
b|z_k|^{\frac{4+\alpha}{3-m}}
=
\Phi_k(0)$ 
on the degeneracy set.

By degeneracy set comparison principle, if $u\leq\Phi_k$ on $\partial B_r$, then $u\leq\Phi_k$ in the entire $B_r$, which however leads to a contradiction since $u(z_k)>\Phi_k(z_k)$.

Hence we have
\begin{equation*}
\sup_{\partial B_{r}}u
\ge
b|z_k|^{\frac{4+\alpha}{3-m}}
+
\eta \left(
\frac{(3-m)^4}
{(4+\alpha)^3(1+\alpha+m)}
\right)^{\frac1{3-m}}
(r-|z_k|)^{\frac{4+\alpha}{3-m}}.
\end{equation*}

Letting $k\to\infty$, it holds that
\begin{equation*}
\sup_{\partial B_{r}}u
\ge
~\eta \left(
\frac{(3-m)^4}
{(4+\alpha)^3(1+\alpha+m)}
\right)^{\frac1{3-m}}
\cdot
~r^{\frac{4+\alpha}{3-m}}
\end{equation*}
for $r\in(0,r_\eta)$.

Consequently since $\eta\in(0,1)$ is arbitrary, letting $\eta\uparrow1$ yields
\[
\liminf_{r\downarrow0}
\frac{\sup_{\partial B_r}u}
{r^{\frac{4+\alpha}{3-m}}}
\ge
C_{\alpha,m}.
\]
\end{proof}

\begin{remark}
As we can see, the local comparison argument around positive-phase points used in \textbf{Step 1} can actually be modified to the proof of the homogeneous case. In that case, the same growth rate will still be obtained, although the resulting lower bound constant is not sharp. 

To be more specific, for $x_0\in\{u>0\}$, working in the ball $B_{\theta|x_0|}(x_0)$, with $0<\theta<1$, then the weight function has a strictly positive inf and we can yield  
$|x|^\alpha
\ge
(1-\theta)^\alpha|x_0|^\alpha$.

Thus for 
$V(x)
=
s|x_0|^{\frac{\alpha}{3-m}}
|x-x_0|^{\frac4{3-m}}$, we can have
\[
c_{\mathrm{loc}}(\theta)
=
\left[
\frac{(3-m)^4}{64(1+m)}
\frac{\theta^4(1-\theta)^\alpha}
{(1+\theta)^{4+\alpha}}
\right]^{\frac1{3-m}}<C_{\alpha,m},
\]
thus the constant $C_{\alpha,m}$ is not an artifact in the comparison procedure. It is a reflection of the structure of the equation. In the homogeneous case, this coefficient in the non-degeneracy estimate, in the case of a source term is present, however, we should do more work to recover the sharp constant estimate.

In fact, if we don't care about the constant, \textbf{Step 2} can be substantially simplified. Indeed, after the argument provides a sequence $z_k$ satisfying $u(z_k)\ge c_0|z_k|^{\frac{4+\alpha}{3-m}}$ for some $c_0>0$, we may consider a simpler function
\begin{equation*}
\Psi(x)=c|x|^{\frac{4+\alpha}{3-m}}
\end{equation*}
where $0<c<\min\{c_0,C_{\alpha,m}\}$.

A direct computation gives that
\begin{equation*}
\Delta_\infty\Psi-|x|^\alpha\Psi^m
=
-c^m
\left[
1-\left(\frac{c}{C_{\alpha,m}}\right)^{3-m}
\right]
|x|^{\frac{3\alpha+4m}{3-m}}.
\end{equation*}

Hence, if $h(x)\ge
-\mu |x|^{\frac{3\alpha+4m}{3-m}}$ with $\mu>0$ sufficiently small, then $\Psi$ is a supersolution. The following argument is similar.

 Therefore, if only a quantitative non-degeneracy estimate is required, the auxiliary function we demonstrated just now provides a considerably shorter argument. The more exact function used in \textbf{Step 2}, however, provides a method to retain the asymptotically sharp homogeneous coefficient and construct a bridge between the homogeneous and inhomogeneous cases.

\end{remark}

\section{Final comments and possible generalizations}

In the final section, we plan to briefly discuss several questions and possible generalizations suggested by the conclusions obtained above. In the future work, we may focus on the global uniqueness of the Dirichlet problem with degenerate weight function, the interaction between the free boundary and the degeneracy set of the weight function, and the geometry of the corresponding free boundary and dead core. Also, we may extend the present analysis from degenerate weight functions to singular ones, as in broader Hardy-H\'enon weight.

The first issue we are interested in is the global uniqueness of the Dirichlet problem. The degeneracy set comparison principle this article developed demonstrates that once the value of $u$ on $F$ is given, the solutions with the same boundary datum are uniquely determined. However, the global uniqueness has not been solved yet. From the comparison principle perspective, $F$ behaves in some sense as an additional interior boundary. Combined with the slice uniqueness, a natural thought is that, in analogy with classical removability questions for harmonic functions, can the value of $u$ inside $F$ actually be determined by the outer boundary values? But we also need to point out that it is different, since $F$ here is intrinsically generated by the vanishing of the weight.

A second direction is to replace the degeneracy set $F=\{0\}$ here by a general closed set. The essence of the object we focus in this paper is actually $F\cap \partial \{u>0\}$. Distance type weights have already appeared in the existing literature, see for instance \cite{BezerraDaSilvaNascimentoSa2024}. Actually, the free boundary can be viewed as $\partial\{u>0\}
=\bigl(\partial\{u>0\}\setminus F\bigr) \cup \bigl(F\cap\partial\{u>0\}\bigr)$. Inspired by the conclusions, we note that even on the free boundary, different points can have different growth rates, and the classification may be $\partial\{u>0\}\setminus F$ and $F\cap \partial \{u>0\}$. Moreover, like the existing literature, since we already get sharp growth regularity and non-degeneracy of the solution near the free boundary intersects with the degeneracy set, we may have similar corollaries and propositions. 

Another related problem is the geometry of the dead core set. Since this article requires the origin falls on the free boundary, the geometric structure here is quite different from that studied in \cite{AraujoLeitaoTeixeira2016}, featuring the symmetric solution structure. In particular, the possible occurrence of non-radial or off-centered dead cores and their interaction with $F$ seem to be natural questions beyond existing results.

Last but not least, another extension of our problem is to return to a broader Hardy-H\'enon weight function. \cite{BezerraDaSilvaNascimentoSa2024} and \cite{WangJiang2026} allow both degenerate and singular case. However, when $\alpha<0$, The origin is a singular point of the weight function rather than a degenerate point. Whether the tools developed here still apply in that case needs extra argument, which is another natural continuation.

\bibliographystyle{amsplain}
\bibliography{refs}

\end{document}